\documentclass[onecolumn,10pt,aps,nofootinbib,superscriptaddress,tightenlines]{elsarticle}

\usepackage{amsmath}
\usepackage{amssymb}
\usepackage{amsthm}
\usepackage{amsfonts}
\usepackage{graphicx}
\usepackage{dsfont}
\usepackage{bbm}
\usepackage{caption}
\usepackage{subcaption}

\newcommand{\abs}[1]{\ensuremath{|#1|}}

\newcommand{\Abs}[1]{\ensuremath{\left|#1\right|}}

\newcommand{\norm}[2]{\ensuremath{|\!|#1|\!|_{#2}}}

\newcommand{\Norm}[2]{\ensuremath{\left|\!\left|#1\right|\!\right|_{#2}}}

\newcommand{\id}{\ensuremath{\mathds{1}}}

\newcommand{\cB}{\mathcal{B}}

\newcommand{\cM}{\mathcal{M}}

\def\beq{\begin{equation}}
\def\eeq{\end{equation}}
\def\bq{\begin{quote}}
\def\eq{\end{quote}}
\def\ben{\begin{enumerate}}
\def\een{\end{enumerate}}
\def\bit{\begin{itemize}}
\def\eit{\end{itemize}}

\def\ra{\rightarrow}

\def\lb{\left(}
\def\rb{\right)}
\def\lset{\lbrace}
\def\rset{\rbrace}

\def\l|{\left|}
\def\r|{\right|}
\def\lbr{\left[}
\def\rbr{\right]}

\def\one{\id}

\newcommand\C{\mathbbm{C}}

\newcommand{\brac}[1]{\left(#1\right)}
\newcommand{\Ibrac}[1]{\left[#1\right]}

\newcommand{\polys}{\mathfrak{P}}
\newcommand{\blaschkes}{\mathfrak{B}}
\newcommand{\disc}{\mathds{D}}

\theoremstyle{plain}
\newtheorem{lemma}{Lemma}
\newtheorem{theorem}[lemma]{Theorem}
\newtheorem{corollary}[lemma]{Corollary}

\theoremstyle{definition}

\begin{document}
\title{Spectral Variation Bounds in Hyperbolic Geometry}
\author{Alexander M\"uller-Hermes}
\ead{muellerh@ma.tum.de}
\address{Zentrum Mathematik, Technische Universit\"{a}t M\"{u}nchen, 85748 Garching, Germany}

\author{Oleg Szehr}
\ead{oleg.szehr@posteo.de}
\address{Zentrum Mathematik, Technische Universit\"{a}t M\"{u}nchen, 85748 Garching, Germany}
\address{Centre for Mathematical Sciences, University of Cambridge, Cambridge CB3 0WA, United Kingdom}

\date{\today}
\begin{abstract}
We derive new estimates for distances between optimal matchings of eigenvalues of non-normal matrices in terms of the norm of their difference. We introduce and estimate a hyperbolic metric analogue of the classical spectral-variation distance. The result yields a qualitatively new and simple characterization of the localization of eigenvalues. Our bound improves on the best classical spectral-variation bounds due to Krause if the distance of matrices is sufficiently small and is sharp for asymptotically large matrices. Our approach is based on the theory of model operators, which provides us with strong resolvent estimates. The latter naturally lead to a Chebychev-type interpolation problem with finite Blaschke products, which can be solved explicitly and gives stronger bounds than the classical Chebychev interpolation with polynomials. As compared to the classical approach our method does not rely on Hadamard's inequality and immediately generalizes to algebraic operators on Hilbert space. \hfill \\
\hfill \\
\textbf{Keywords:} Spectral variation bounds, Blaschke products, hyperbolic geometry, resolvent bounds \hfill \\
\hfill \\
2010 MSC 15A60 ,15A18, 15A42, 65F35

\end{abstract}

\maketitle

%
%
%
%
\section{Introduction} 

For arbitrary complex $n\times n$-matrices $A,B\in\cM_n$ we study distances of optimal matchings of their spectra $\sigma(A),\sigma(B)$. The (Euclidean) \emph{optimal matching distance}~\cite{Bhatia1} of two sets $\lset a_i\rset^n_{i=1},\lset b_i\rset^n_{i=1}\subset \C$ is defined as
\begin{align*}
d_E\lb \lset a_i\rset,\lset b_i\rset\rb = \min_{\sigma\in S_n}\max_{1\leq i\leq n} \Abs{a_i - b_{\sigma\lb i\rb}},
\end{align*}
where $S_n$ denotes the group of permutations of $n$ objects. A prototypical spectral variation bound in terms of this distance is of the form
\begin{align}
d_E\lb\sigma\lb A\rb,\sigma\lb B\rb\rb \leq C_n\lb\Norm{A}{}+\Norm{B}{}\rb^{1-\frac{1}{n}}\Norm{A-B}{}^{\frac{1}{n}},
\label{equ:GenBound}
\end{align}
where $C_n$ can only depend on $n$ and $\Norm{\cdot}{}$ denotes the usual operator norm.
Such estimates have been studied in many articles and books over the last decades, see for example
~\cite{Ostrowski,Friedland,Elsner198577,Phillips1990165,MGil,
Bhatia1990195,Krause199473,StewartSun} and references therein. 
Despite considerable effort the best $C_n$ in~\eqref{equ:GenBound} is still not known, the currently best value seems to be $C_n = \frac{16}{3\sqrt{3}}$~\cite{Krause199473}.

In this work we present a new approach to spectral variation estimates and derive new bounds that characterize the localization of spectra of non-normal matrices. We introduce a (pseudo-) hyperbolic analogue of the optimal matching distance and derive estimates on this quantity in terms of $\Norm{A}{}, \Norm{B}{}, \Norm{A-B}{}$ and $n$. These hyperbolic estimates are generally incomparable to
Equation~\eqref{equ:GenBound} meaning that there are cases, where they perform better than the previously known bounds, but also cases where they do not, see Figure~\ref{fig:CompBh}. As it turns out
we can use the hyperbolic estimates to improve the best value for $C_n$, if $\Norm{A-B}{}$ is \emph{small enough} (Corollary \ref{cor:Improved} below). In the limit of large $n$ our $C_n$ approaches $2$, which is optimal.

\begin{figure}
\centering
\begin{minipage}{.5\textwidth}
  \centering
  \includegraphics[width=6.8cm]{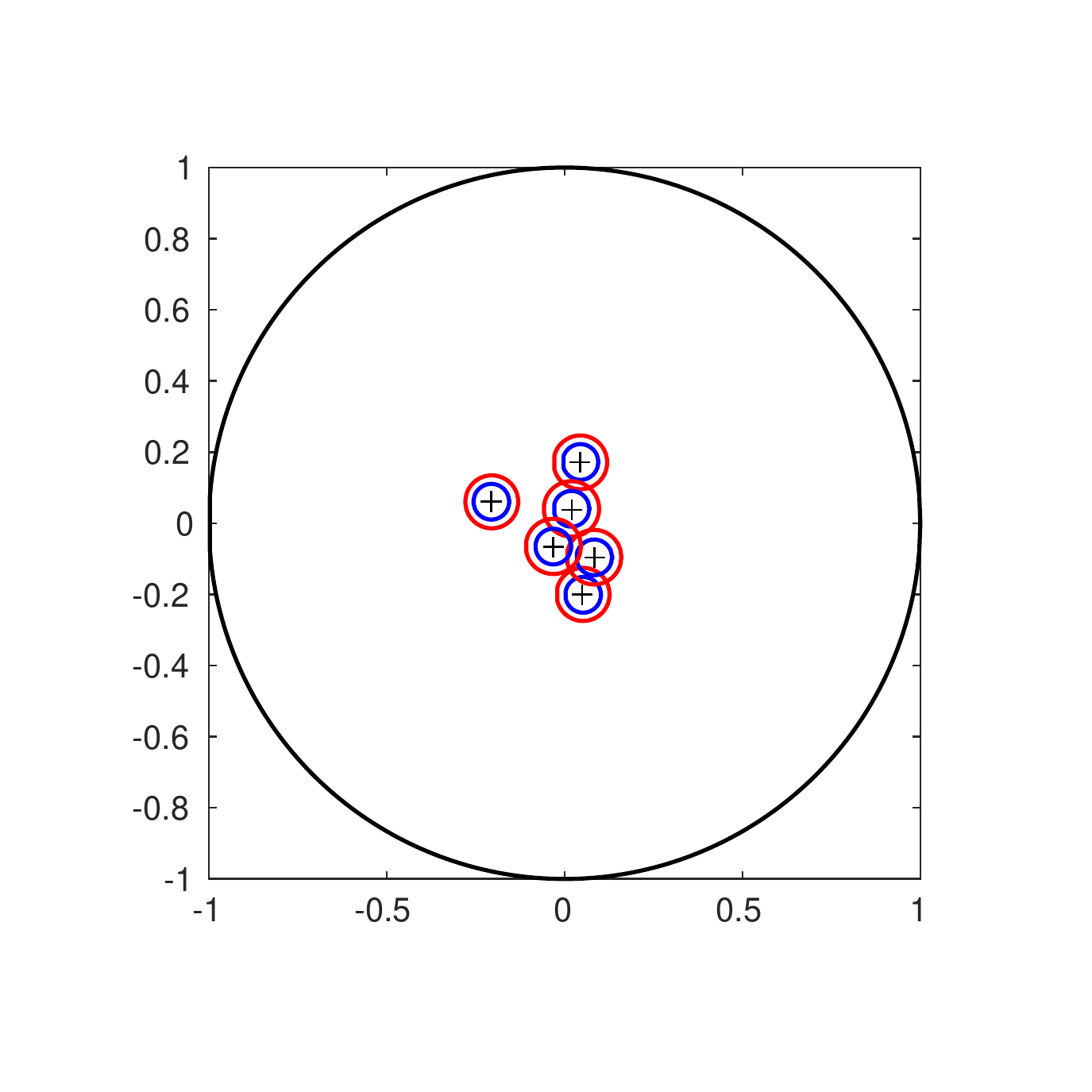}
  \vspace*{-1.5cm}
  \captionof*{figure}{(a) $\norm{A}{}=0.3$}
\end{minipage}%
\begin{minipage}{.5\textwidth}
  \centering
  \includegraphics[width=6.8cm]{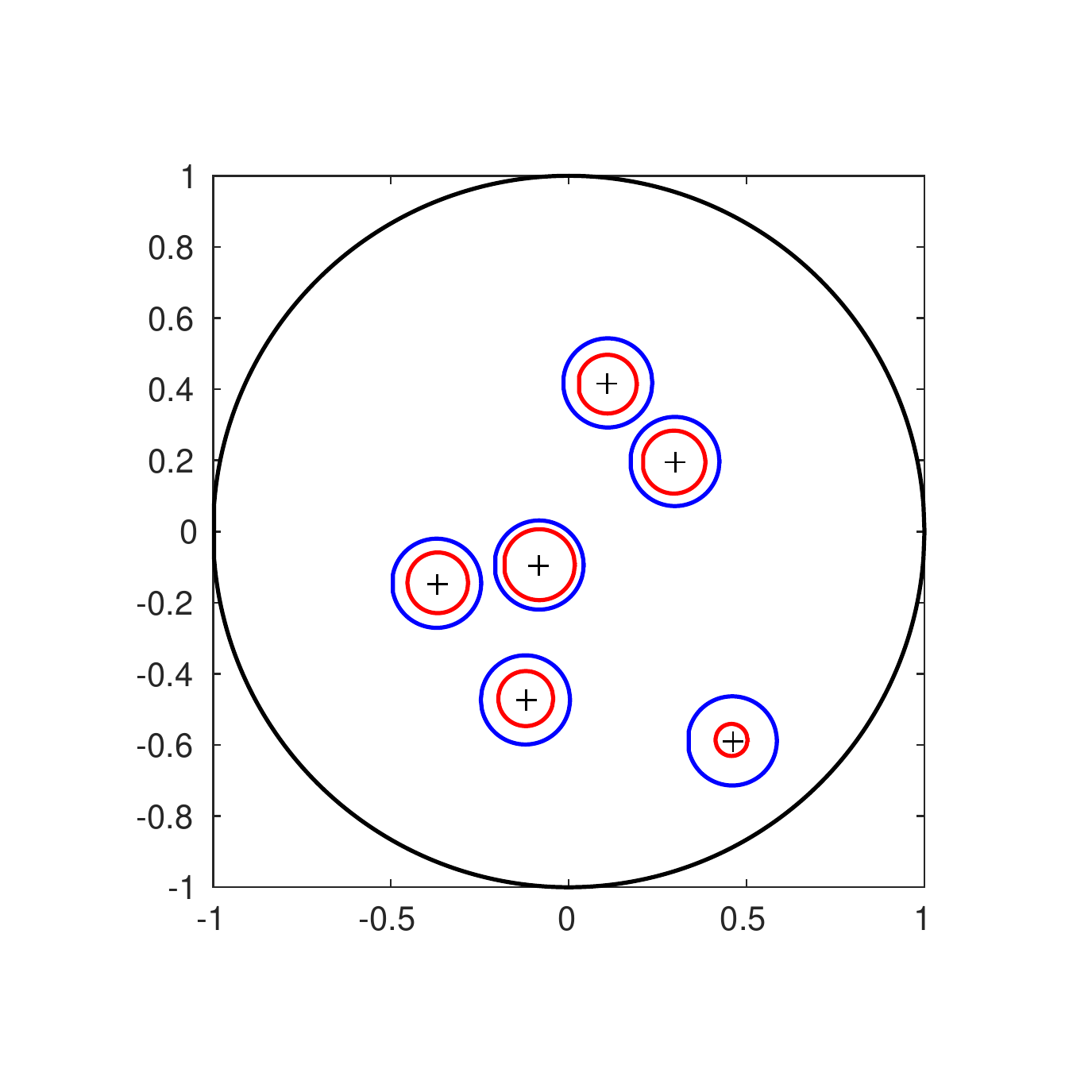}
  \vspace*{-1.5cm}
  \captionof*{figure}{(b) $\norm{A}{}=0.9$}
\end{minipage}
\caption{Localization of eigenvalues of a perturbed $6\times 6$ matrix $B=A+E$ for randomly chosen $A$ and $\norm{E}{}\leq 10^{-10}$ with respect to the spectrum of $A$ (crosses). The circles contain an eigenvalue of $B$. Blue circles are due to the classical estimate \eqref{equ:GenBound} with $C_n = 2^{2-\frac{1}{n}}$, while red circles result from Theorem~\ref{hyper} below.}
\label{fig:CompBh}
\end{figure}

Our argument is guided by a classical approach due to Phillips~\cite{Phillips1990165} that in turn builds on techniques developed by Friedland~\cite{Friedland} and Elsner~\cite{ElsnerAlt}. Phillips reduces the problem of obtaining a good estimate of the form~\eqref{equ:GenBound} to one of \emph{minimizing} the norm of a resolvent along certain paths in the complex plane. The latter can be accomplished using a classical interpolation theorem due to Chebyshev. Phillips' approach was developed further by Bhatia, Elsner and Krause~\cite{Bhatia1990195,Krause199473} who employed a Hadamard-type inequality~\cite{Elsner198577} (equation~\eqref{hada} below) due to Elsner to avoid resolvent estimates. Their approach provided a better estimate for $C_n$. 

On the technical side, our article contains two key innovations to the methods developed in the cited publications. First, we employ recent spectral resolvent estimates~\cite{OlegRes,Nikolski2006} that are stronger than the Hadamard-type inequality~\eqref{hada} used by Bhatia, Elsner and Krause. These resolvent estimates are derived using an interpolation-theoretic approach to eigenvalue bounds introduced in~\cite{Nikolski2006}. Second, our resolvent estimates naturally lead us to a Chebyshev-type interpolation problem with finite Blaschke products that yields quantitatively better estimates as compared to the analogous classical Chebyshev interpolation problem for polynomials. The solution to this interpolation problem was recently provided in~\cite{Tsang2012} in terms of the so called \emph{Chebychev Blaschke products}. Although Chebychev-Blaschke interpolation occurs naturally in our context it turns out that it leads to a rather small numerical advantage and seems to be of rather theoretic interest. Our methods immediately generalize to algebraic operators on Hilbert/Banach space, which allows us to improve on the study of such operators in~\cite{Abdel1,Abdel2}.

This article is structured as follows. In Section~\ref{classm} we recapitulate the classical approach due to Bhatia, Elsner, Krause and Phillips. Section~\ref{basichyper} introduces some basic notation from hyperbolic geometry. In Section~\ref{newm} we state some results from a model operator theoretic approach to spectral estimates as well as the aforementioned theorem about interpolation with finite Blaschke products. Section~\ref{hyperresult} contains our main results. In Section~\ref{discussion} we compare our estimates to the ones of equation~\eqref{equ:GenBound}.

\section{Preliminaries}

\subsection{Classical methods for spectral variation bounds}\label{classm}

The proofs of spectral variation bounds as in formula~\eqref{equ:GenBound} by Bhatia, Elsner, Krause and Phillips as well as our derivation have a similar core. For $A,B\in\cM_n$ the eigenvalues of the convex combination
\begin{align*}
A_t = (1-t)A + tB
\end{align*}
trace $n$ continuous curves in the complex plane as $t$ varies from $0$ to $1$~\cite{Bhatia1}. These curves connect the eigenvalue sets $\sigma(A)$ and $\sigma(B)$ and establish a matching even though they might intersect. It has been shown by Elsner~\cite{Elsner198577,Bhatia1} that if $X,Y\in\cM_n$ and $z$ is an eigenvalue of $Y$, then
\begin{align}
\Abs{\textnormal{det}(z\id-X)}\leq\Norm{X-Y}{}(\Norm{X}{}+\Norm{Y}{})^{n-1}.\label{hada}
\end{align}
It follows that along any particular curve $z:\Ibrac{0,1}\ra\C$ with $z(0)=a\in\sigma\lb A\rb$ and $z(1)=b\in\sigma\lb B\rb$ we have
\begin{align*}
\Abs{\textnormal{det}( z\brac{t}\one - A)} \leq \lb\Norm{A}{}+\Norm{B}{}\rb^{n-1}\Norm{A-B}{}.
\end{align*}
Thus, in order to prove the bound~\eqref{equ:GenBound} it is sufficient to find $t^*\in[0,1]$ such that
\begin{align*}
\frac{1}{(C_n)^n}\Abs{a-b}^n\leq |\textnormal{det}\lb z\brac{t^*}\one - A\rb |
\end{align*}
holds. The validity of the above inequality with $C_n= 2^{2-\frac{1}{n}}$ is ensured by a Chebychev-type interpolation result. If $z$ is a continuous curve with endpoints $a$ and $b$ then
\begin{align}
\min_{p\in\polys_n}\max_{t\in\Ibrac{0,1}}|p(z(t))| \geq \frac{\abs{b-a}^n}{2^{2n-1}},
\label{equ:Cheb}
\end{align}
where $\mathfrak{P}_n$ denotes the set of monic polynomials of degree $n$~\cite[Lemma VIII.1.4]{Bhatia1}.
%
Equation~\eqref{equ:Cheb} follows from the well-known fact, that among the real polynomials of degree $n$ with leading coefficient $1$,
the normalized Chebychev polynomial is the one whose maximal absolute value on the interval $[-1,1]$ is minimal~\cite[p. 31]{Rivlin1981}. We refer to~\cite{Bhatia1}, Chapter~VIII for a more detailed discussion of the above derivation.

In earlier work Phillips~\cite{Phillips1990165} did not employ~\eqref{hada} but instead he relied on a Bauer-Fike estimate~\cite{bauerfiker}, which asserts that for $z\in\sigma\lb Y\rb\setminus\sigma\lb X\rb$ and $X,Y\in\cM_n$ we have
\begin{align}
\frac{1}{\Norm{X-Y}{}}\leq \Norm{(z\one - X)^{-1}}{}.
\label{equ:inequ}
\end{align}
A suitable estimate for the occurring resolvent and equation~\eqref{equ:Cheb} prove~\eqref{equ:GenBound}. The prefactor $C_n$ obtained by Phillips in this way is however slightly worse. As we shall see (cf.~Section~\ref{newm}), the original estimate~\eqref{equ:inequ} in fact \emph{is stronger} than the inequality in~\eqref{hada} and it is possible to prove~\eqref{equ:GenBound} with $C_n=2^{2-\frac{1}{n}}$ starting from~\eqref{equ:inequ}. To this end we bound the resolvent using advanced methods from the theory of model operators~\cite{Nikolski2006,OlegRes}. These techniques naturally lead us to estimate a hyperbolic analogue of the optimal matching distance and to the more sophisticated interpolation with finite Blaschke products.

\subsection{Basics from Hyperbolic geometry}\label{basichyper}

We denote by $\mathbb{D}$ the open unit disk in the complex plane and by $\overline{\mathbb{D}}$ its closure. For $x,y\in\overline{\mathbb{D}}$ the (pseudo-)hyperbolic distance~\cite{Garnett81} is
\begin{align*}
p(a,b)=\Abs{\frac{a-b}{1-\bar{a}b}}.
\end{align*}
It is not hard to verify that $p$ is symmetric, satisfies a triangle inequality and that $0\leq p(a,b)\leq1$, see \cite{Garnett81}. The \lq\lq hyperbolic disk \rq\rq around $a$ with radius $r\leq1$, i.e.~the set $\{z\:|\: p(a,z)<r\}$, is also an Euclidean disk with center $C$ and radius $R$ given by~\cite[Chapter 1]{Garnett81}
\begin{align}
C=\frac{1-r^2}{1-r^2 \abs{a}^2}a\quad\textnormal{and}\quad R=\frac{1-\abs{a}^2}{1-r^2 \abs{a}^2}r.
\label{equ:HypDisc}
\end{align}
We study optimal matchings of spectra with respect to hyperbolic distance. For two sets $\lset a_i\rset^n_{i=1},\lset b_i\rset^n_{i=1}\subset \mathbb{D}$ we define the hyperbolic optimal matching distance as
\begin{align}
d_H\lb \lset a_i\rset,\lset b_i\rset\rb = \min_{\sigma\in S_n}\max_{1\leq i\leq n} \Abs{\frac{a_i - b_{\sigma\lb i\rb}}{1-\bar{a}_ib_{\sigma\lb i\rb}}},
\label{hyperdist}
\end{align}
where $S_n$ denotes the group of permutations of $n$ objects. We assume that $A,B\in\cM_n$ have spectra $\sigma\lb A\rb,\sigma\lb B\rb\subset \overline{\mathbb{D}}$, which can always be achieved by a suitable normalization. If $d_H\lb\sigma\lb A\rb,\sigma\lb B\rb\rb$ is bounded by $r$ it follows that in a hyperbolic disk of radius $r$ around an eigenvalue $a$ of $A$ there is an eigenvalue $b$ of $B$. Furthermore, $a$ and $b$ are contained in an Euclidean disk of radius $R$ and center $C$. Thus the hyperbolic estimate entails an Euclidean characterization of the localization of eigenvalues. We will discuss this further in Section~\ref{discussion}.

\subsection{New methods for spectral variation bounds} \label{newm}

This section contains two advanced results that we require for our analysis. The first one allows to estimate norms of rational functions of matrices in terms of their eigenvalues and builds on deep results from harmonic analysis and operator theory. The second one is the aforementioned analogue of the Chebychev interpolation problem for Blaschke products and is rooted in the theory of elliptic functions.

Consider a matrix $A\in\cM_n$ with $\norm{A}{}\leq1$ and minimal polynomial $m=\prod_{i=1}^{\abs{m}}(z-\lambda_i)$, where $\abs{m}$ denotes the degree of $m$. The problem of finding a spectral estimate on the operator norm of a rational function of $A$ has a complete solution. It is sufficient to consider a certain \emph{model matrix} $M_m$ that is associated to $m$. The latter is a lower-triangular $\abs{m}\times\abs{m}$ matrix and is entry-wise given by
\begin{align}
\left(M_m\right)_{ij}=\begin{cases}\qquad\qquad\qquad{0}\ &if \ i<j\\
\qquad\qquad\qquad\lambda_i\ &if\ i=j\\
(1-\abs{\lambda_i}^2)^{1/2}(1-\abs{\lambda_j}^2)^{1/2}\prod_{\mu=j+1}^{i-1}\left(-\bar{\lambda}_\mu\right)\ &if \ i>j
\end{cases}.\label{modelmat}
\end{align}
The algebraic multiplicity of the eigenvalue $\lambda_i$ of $M_m$ is exactly the number of factors associated to $\lambda_i$ in the minimal polynomial of $A$.

\begin{lemma}[\cite{Nikolski2006,OlegRes}]\label{ratmod} Let $A\in\cM_n$ with operator norm $\norm{A}{}\leq1$ and minimal polynomial $m=\prod_{i=1}^{\abs{m}}(z-\lambda_i)$. Let $\psi$ be a rational function whose set of poles does not intersect $\sigma(A)$. For the associated model matrix~\eqref{modelmat} it holds that
$\Norm{M_m}{}\leq1$ and $\sigma(M_m)=\sigma(A)$ and
$$\norm{\psi(A)}{}\leq\norm{\psi(M_m)}{}.$$
\end{lemma}
This lemma is a consequence of an interpolation-theoretic approach to eigenvalue estimates, which has been established in~\cite{Nikolski2006}. The occurrence of model matrices can be seen as a result of Sarason's approach to interpolation theory~\cite{Sarason} or the commutant lifting theorem of Nagy-Foia\c{s}~\cite{NFinter,FFinter}. Note that the above estimate is achieved by $M_m$ and hence Lemma~\ref{ratmod} provides a complete solution to the problem of finding a spectral bound to $\norm{\psi(A)}{}$.
It is not hard to verify that Lemma~\ref{ratmod} implies that for any $X\in\cM_n$ with $\Norm{X}{}\leq1$ and minimal polynomial $m=\prod_{i=1}^{\abs{m}}(z-\lambda_i)$ it holds that~\cite[Theorem 3.12]{Nikolski2006}
\begin{align}
\Norm{X^{-1}}{}\leq\prod_{i=1}^{\abs{m}}\frac{1}{\abs{\lambda_i}}.\label{inverse}
\end{align}

In order to prove a hyperbolic spectral variation estimate we will replace the Chebychev-type interpolation result~\eqref{equ:Cheb} by an interpolation theorem for finite Blaschke products. This result heavily relies on the theory of Jacobi Theta functions, cf.~\cite{Tsang2012}. We abstain from going into details of this theory and instead we just define the corresponding functions on a restricted domain. For some $q\in[0,1)$ we set
\begin{align*}
\vartheta_2\brac{q} := \sum^\infty_{k=-\infty} q^{\brac{k+\frac{1}{2}}^2}\qquad\textnormal{and}\qquad
\vartheta_3\brac{q} := \sum^\infty_{k=-\infty} q^{k^2}.
\end{align*}
For $q\in (0,1)$ we define the elliptic modulus as
\begin{align}
k\brac{q} := \brac{\frac{\vartheta_2\brac{q}}{\vartheta_3\brac{q}}}^2.
\label{equ:k}
\end{align}
By continuous extension we set $k(0):=0$ and $k(1):=1$ and note that $k(q)$ is strictly increasing in $q$ ~\cite[Section 21.7]{WatsonWhit}. 
A finite Blaschke product is a product of the form
\begin{align}
B(z)=\prod_{i=1}^n\frac{z-\lambda_i}{1-\bar{\lambda}_i z},\label{Blascker}
\end{align}
where $\{\lambda_i\}_{i=1,...,n}\subset\overline{\mathbb{D}}$. 
\begin{lemma}[\cite{Tsang2012}, p. 32]
Let $\blaschkes_n$ denote the set of finite Blaschke products of the form of Equation~\eqref{Blascker} with coefficients $\{\lambda_i\}_{i=1,...,n}\subset\disc$. We have that 
\begin{align*}
\min_{B_n\in\blaschkes_n}\max_{z\in\Ibrac{-\sqrt{k\brac{q}},\sqrt{k\brac{q}}}} |B_n\brac{z}| = \sqrt{k\brac{q^n}},
\end{align*}
for $q\in\lbr 0,1\rbr$ and where $k(q)$ is as in Equation~\eqref{equ:k}.
\label{lem:BlaschkeInt}
\end{lemma}
In~\cite{Tsang2012} an optimal Blaschke product achieving the optimization in Lemma~\ref{lem:BlaschkeInt} for given $n$ is given explicitly and referred to as the \emph{Chebychev Blaschke product}.

\section{Spectral variation bounds}\label{hyperresult}

\subsection{Euclidean spectral variation bounds}

We begin with a new proof of inequality~\eqref{hada} based on inequality~\eqref{equ:inequ}. We improve~\eqref{hada} in that that the degree $\abs{m}$ of the minimal polynomial of $A\in\cM_n$ occurs instead of the dimension $n$. This leads to the following improvement of the Euclidean spectral variation bound from~\cite{Bhatia1990195}.

\begin{theorem}[Euclidean spectral variation]\label{euclid}
Let $A,B\in\cM_n$ and let $\abs{m}$ denote the degree of the minimal polynomial of $A$. Then
\begin{align*}
d_E\lb\sigma\lb A\rb,\sigma\lb B\rb\rb \leq 2^{2-\frac{1}{\abs{m}}}\lb\Norm{A}{}+\Norm{B}{}\rb^{1-\frac{1}{\abs{m}}}\Norm{A-B}{}^{\frac{1}{\abs{m}}}.
\end{align*}
\end{theorem}
\begin{proof}
We strengthen the Hadamard type inequality~\eqref{hada} (see~\cite[Section 2]{Elsner198577} for the original derivation). For any $X,Y\in\cM_n$ combining inequalities~\eqref{equ:inequ} and~\eqref{inverse} yields
\begin{align*}
\frac{1}{\Norm{X-Y}{}}\leq\Norm{(z\id-X)^{-1}}{}\leq\frac{\Norm{z\id-X}{}^{\abs{m}-1}}{\prod_{i=1}^{\abs{m}}\abs{z-\lambda_i}}\leq\frac{(\Norm{X}{}+\Norm{Y}{})^{\abs{m}-1}}{\prod_{i=1}^{\abs{m}}\abs{z-\lambda_i}},
\end{align*}
where $\abs{m}$ denotes the degree of the minimal polynomial of $X$. Apart from this the proof of Theorem~\ref{euclid} follows the line of~\cite{Bhatia1990195}.
\end{proof}
In~\cite{Abdel1} spectral variation bounds for algebraic elements of unital Banach algebras were studied. Our estimate also holds for algebraic operators on Hilbert spaces~\cite{Nikolski2006} and is stronger than~\cite[Theorem 2.2.52]{Abdel1}. For algebraic elements of general Banach algebras, one can improve~\cite[Theorem 2.2.52]{Abdel1} using~\cite[Theorem~3.20]{Nikolski2006} instead of~\eqref{inverse} and the same derivation as above.

\subsection{Hyperbolic spectral variation bounds}
In this section we prove our main result, which is a spectral variation estimate
for non-normal matrices in hyperbolic geometry.

\begin{theorem}[Hyperbolic spectral variation]\label{hyper}
Let $A,B\in\cM_n$ with $\Norm{A}{},\Norm{B}{}<1$, let $\abs{m}$ denote the degree of the minimal polynomial of $A$ and let $\rho(B)\leq\Norm{B}{}$ denote the spectral radius of $B$. Then
\begin{align*}
d_H\lb\sigma\lb A\rb,\sigma\lb B\rb\rb &\leq  k\left(k^{-1}\left(\frac{\Norm{A-B}{}^2}{\lb 1-\rho\lb  B\rb\Norm{A}{}\rb^2}\right)^{\frac{1}{2\abs{m}}}\right)\\
&\leq
\frac{2^{2-\frac{1}{\abs{m}}}}{\lb 1-\rho\lb  B\rb\Norm{A}{}\rb^{\frac{1}{\abs{m}}}}\Norm{A-B}{}^{\frac{1}{\abs{m}}},
\end{align*}
where $d_H$ is defined in \eqref{hyperdist} and $k$ is defined in \eqref{equ:k}.
\end{theorem}
The assumption $\Norm{A}{},\Norm{B}{}<1$ is not principal as it can be achieved by a suitable normalization, cf.~Section~\ref{discussion}.
Note that due to the definition of $k$ and $k^{-1}$ in terms of power series, their values can be efficiently computed using mathematical software~\cite[Section XXI.21.8]{WatsonWhit}. Hence, the first bound in the theorem can be computed without any difficulty. The second inequality provides a simplification of the first estimate and is proven in Lemma~\ref{hasest} below. 

To prove Theorem~\ref{hyper}, we first provide a natural hyperbolic analogue of the interpolation result~\eqref{equ:Cheb} (Lemma VIII.1.4 in~\cite{Bhatia1}).
\begin{lemma}\label{irgendson} Let $z:[0,1]\mapsto\overline{\mathbb{D}}$ be a continuous curve with endpoints $z(0)=a$ and $z(1)=b$ and let $\cB_n$ be the set of finite Blaschke products with coefficients in $\mathbb{D}$. Choose $q\in[0,1]$ such that $k(q)=\Abs{\frac{a-b}{1-\bar{a}b}}$. Then
\begin{align*}
\min_{B_n\in\blaschkes_n}\max_{t\in[0,1]} |B_n\brac{z(t)}| \geq \sqrt{k\brac{q^{2n}}}.
\end{align*}
\end{lemma}

\begin{proof}

We would like to apply Lemma~\ref{lem:BlaschkeInt}. For this purpose we consider a hyperbolic geodesic curve $\Gamma$ through the points $a,b\in\disc$. This curve can be parametrized as~\cite[Prop. 2.3.17]{Krantz}
\begin{align*}
s\mapsto\Gamma(s)=\frac{sC(a,b) + a}{1+s\bar{a}C(a,b)},\quad C(a,b):=\frac{b-a}{1-\bar{a}b},
\end{align*}
where $s\in\lb-\abs{\frac{1}{C(a,b)}},\abs{\frac{1}{C(a,b)}}\rb$ and $\Gamma(0)=a$, $\Gamma(1)=b$.
The hyperbolic disk carries a natural perpendicular projection onto $\Gamma$. For any $z\in\mathbb{D}$ there is a unique geodesic through $z$ that orthogonally intersects $\Gamma$. We denote by $z'\in \Gamma$ the point obtained by mapping $z$ along this geodesic onto $\Gamma$ and call it the perpendicular projection of $z$ onto $\Gamma$. By using elementary hyperbolic geometry, see Appendix \ref{sec:Appendix}, one can prove that the hyperbolic pseudo-distance is contractive under this projection, i.e. for any $z,w\in\mathbb{D}$
\begin{align*}
\Abs{\frac{z-w}{1-\bar{z}w}}\geq \Abs{\frac{z'-w'}{1-\overline{z'}w'}}.
\end{align*}
For a Blaschke product $B_n(z(t))$ we have
\begin{align*}
\prod^n_{i=1} \Abs{\frac{z\brac{t}-\lambda_i}{1-\overline{z\brac{t}}\lambda_i}} &\geq \prod^n_{i=1} \Abs{\frac{z\brac{t}'-\lambda'_i}{1-\overline{z\brac{t}'}\lambda'_i}}
= \prod^n_{i=1} \Abs{\frac{\Abs{C(a,b)}\brac{s - s_i}}{1-\Abs{C(a,b)}^2s_i s}}.
\end{align*}
The equality follows from Schwarz-Pick lemma~\cite{Garnett81} by choosing
\begin{align*}
s,s_i\in\lb-\left|\frac{1}{C(a,b)}\right|,\left|\frac{1}{C(a,b)}\right|\rb
\end{align*}
so that $\lambda'_i = \Gamma(s_i)$ and $z\brac{t}' = \Gamma(s)$. Consider now $t$ such that $z\brac{t}' = \Gamma(s)$ for some $s\in\lbr 0,1\rbr$ (i.e.~the perpendicular projection maps $z\brac{t}$ to the geodesic arc between $\Gamma(0)=a$ and $\Gamma(1)=b$). It follows by elementary computation that
\begin{align}
&\min_{s_i\in \lb-\abs{\frac{1}{C(a,b)}},\abs{\frac{1}{C(a,b)}}\rb} \max_{s\in \lbr 0,1\rbr} \Abs{\prod^{n}_{i=1}\frac{\Abs{C(a,b)}\brac{s-s_i}}{1-\Abs{C(a,b)}^2ss_i}} \\
&~= \min_{s_i\in \lb -1,1\rb} \max_{s\in\Ibrac{0,\Abs{C(a,b)}}}\Abs{\prod^{n}_{i=1}\frac{s-s_i}{1-ss_i}}\label{pitty}\\
& ~= \min_{s_i\in \lb -1,1\rb} \max_{s\in\Ibrac{-\sqrt{\Abs{C(a,b)}},\sqrt{\Abs{C(a,b)}}}} \Abs{\prod^{n}_{i=1}\frac{s^2-s_i}{1-s^2s_i}}\nonumber\\
&~= \min_{s_i\in \lb -1,1\rb} \max_{s\in\Ibrac{-\sqrt{\Abs{C(a,b)}},\sqrt{\Abs{C(a,b)}}}} \prod^{n}_{i=1}\brac{\Abs{\frac{s-\sqrt{s_i}}{1-s\sqrt{s_i}}}\Abs{\frac{s+\sqrt{s_i}}{1+s\sqrt{s_i}}}}\nonumber.
\end{align}
We can now choose $q\in[0,1]$ with $k(q)=\Abs{C(a,b)}$ and apply Lemma~\ref{lem:BlaschkeInt} to the final Blaschke product of degree $2n$. We conclude that the last term is always bounded from below by $\sqrt{k\brac{q^{2n}}}$.
\end{proof}
%
%
%
%
%
%
%
%
%
%
%
%
%
%
%
%
%
%
%
%
%
%
%
%
%
%
In Lemma~\ref{irgendson} writing out the right-hand side gives $$\sqrt{k\left(k^{-1}\left(\Abs{\frac{a-b}{1-\bar{a}b}}\right)^{2n}\right)},$$ which depends on $\Abs{\frac{a-b}{1-\bar{a}b}}$ in a complicated way. To obtain a simplified expression we might use Lemma~\ref{hasest} below. However, applying Lemma~\ref{hasest} goes at the price of losing the advantage due to Lemma~\ref{lem:BlaschkeInt} in the final bound of Theorem \ref{hyper}. In particular the estimate on the right-hand side of the lemma below can be derived from ordinary Chebychev interpolation~\eqref{equ:Cheb} without relying on Lemma~\ref{lem:BlaschkeInt}, see Remark~\ref{vomTeufel} at the end of the proof of Theorem~\ref{hyper}.
\begin{lemma}\label{hasest}
For the elliptic modulus $k$ as defined in (\ref{equ:k}) and $q\in[0,1]$ we have
\begin{align*}
\sqrt{k\brac{q^n}} \geq \frac{1}{2^{n-1}}\brac{\sqrt{k\brac{q}}}^n.
\end{align*}
\end{lemma}
Figure~\ref{KvonQ} shows the $n$-dependence of the inequality for different values of $q$.

\begin{figure}[t]
\centering
\begin{minipage}{.5\textwidth}
  \centering
  \includegraphics[width=6.5cm]{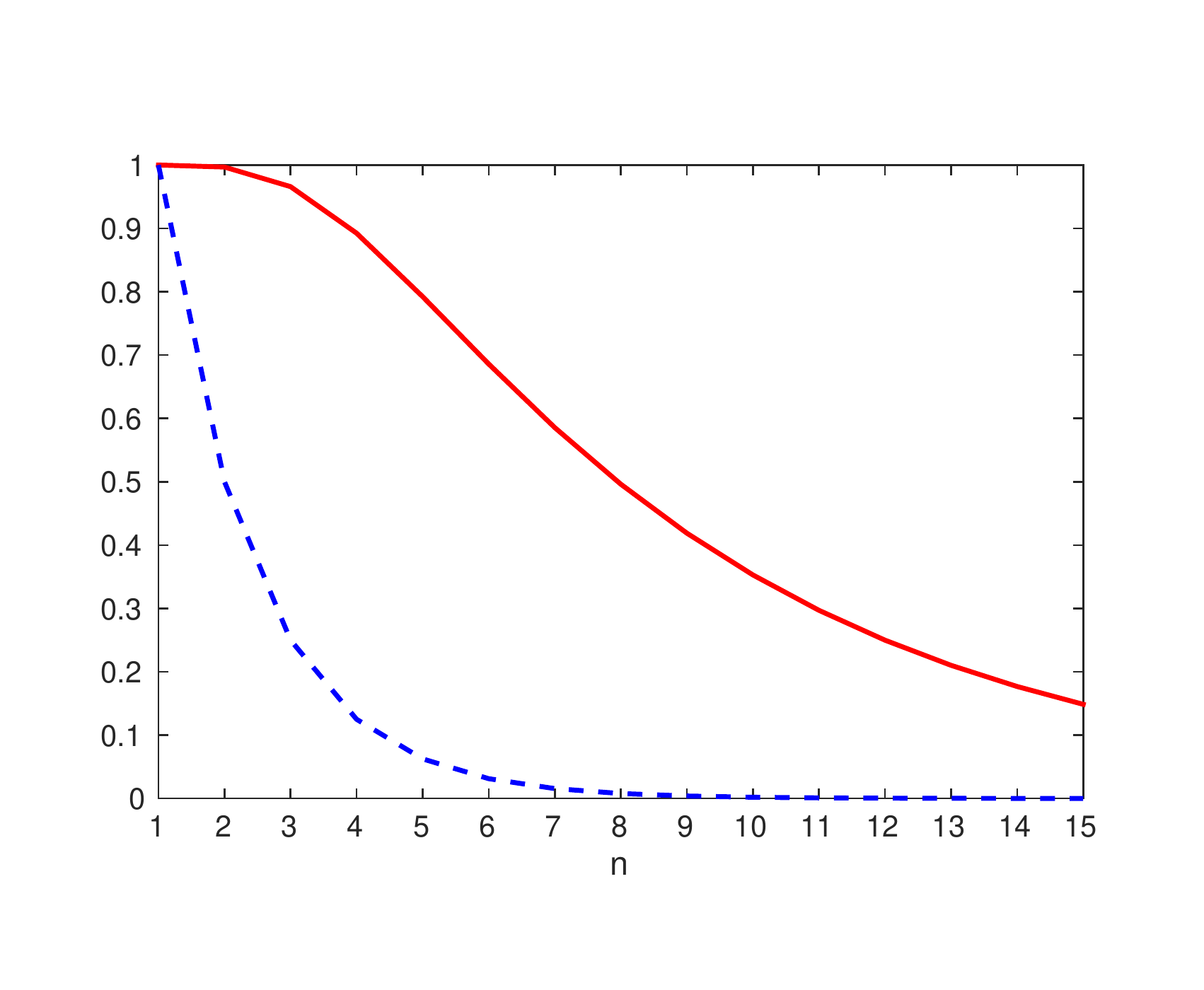}
  \vspace*{-1.2cm}
  \captionof*{figure}{(a) $q=0.5$}
\end{minipage}%
\begin{minipage}{.5\textwidth}
  \centering
  \includegraphics[width=6.5cm]{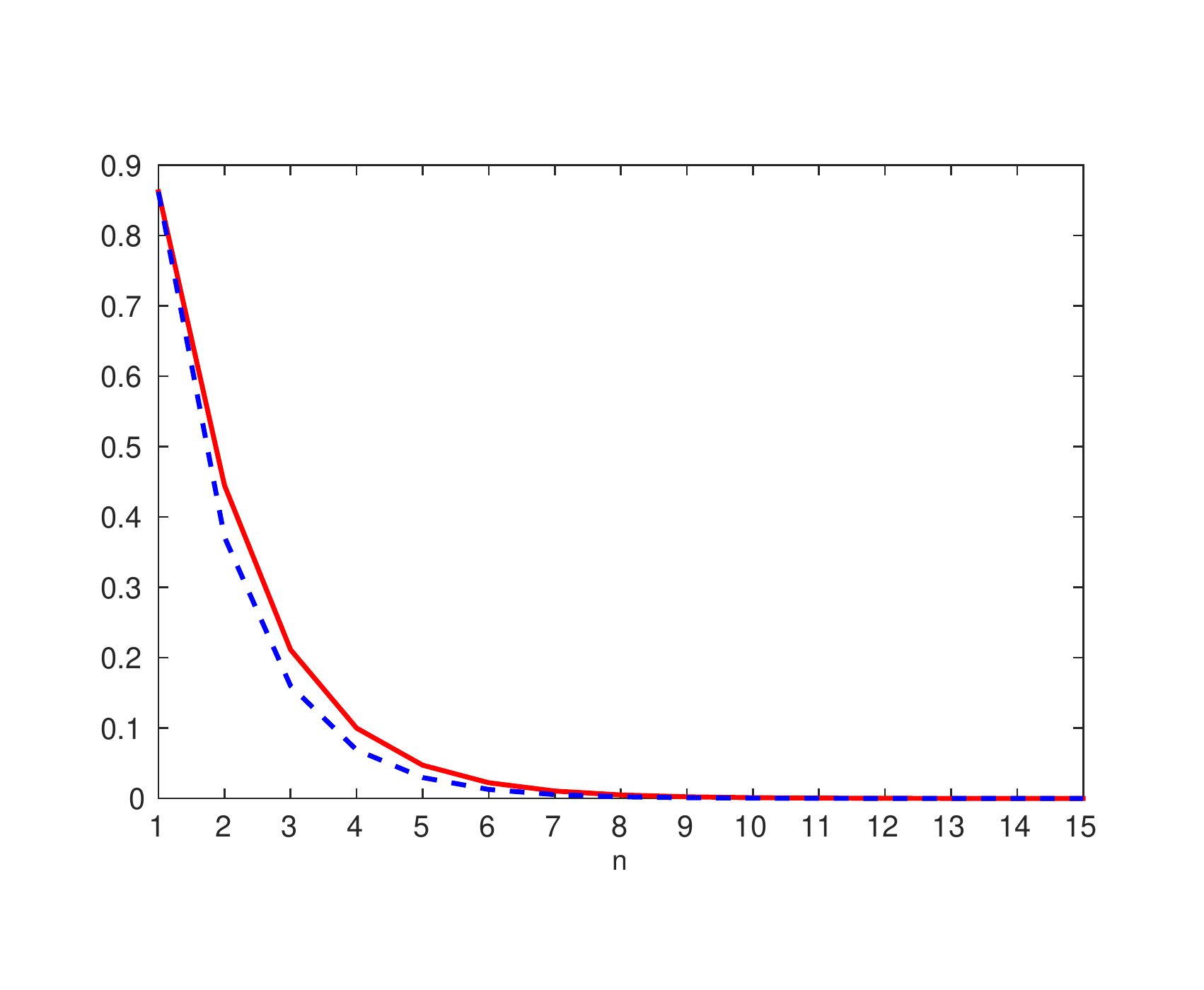}
  \vspace*{-1.2cm}
  \captionof*{figure}{(b) $q=0.05$}
\end{minipage}\\
\begin{minipage}{.5\textwidth}
  \centering
  \includegraphics[width=6.5cm]{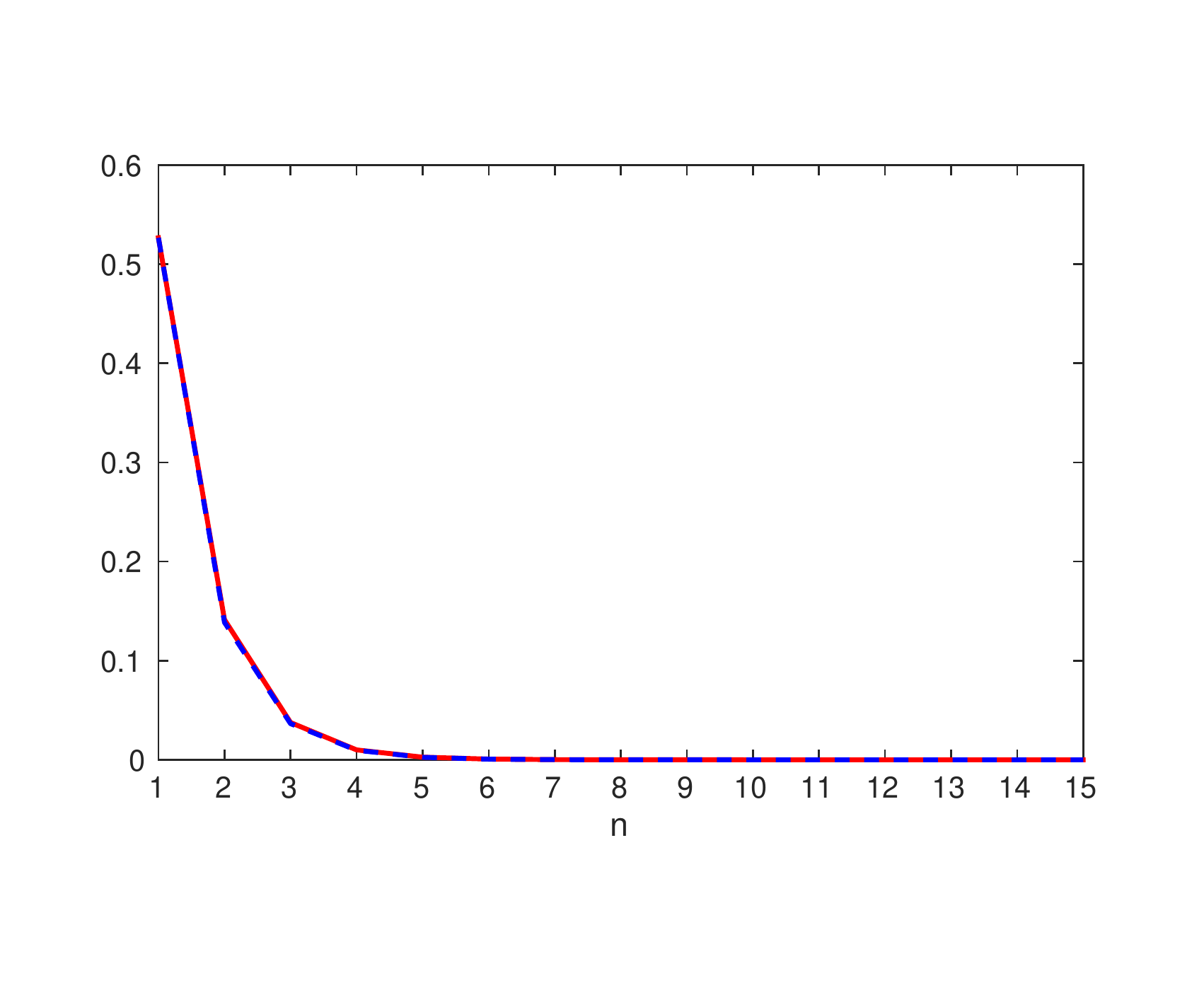}
  \vspace*{-1.4cm}
  \captionof*{figure}{(c) $q=0.005$}
\end{minipage}
\caption{Plot of $n$-dependence of the estimate in Lemma~\ref{hasest}. Red graphs represent $\sqrt{k(q^n)}$ and blue graphs $\frac{1}{2^{n-1}}k(q)^{n/2}$. At $n=1$ equality holds and either function takes the value $\sqrt{k(q)}$.}
\label{KvonQ}
\end{figure}

\begin{proof}
The elliptic modulus admits an infinite product representation \cite[Formula 8.197.3]{SovjetEnc}
\begin{align*}
k(q)=4\sqrt{q}\prod_{k=1}^{\infty}\left(\frac{1+q^{2k}}{1+q^{2k-1}}\right)^4.
\end{align*}
The function
\begin{align*}
x\mapsto\frac{(1+x)^n}{1+x^n}
\end{align*}
is monotonically increasing on the interval $[0,1]$, which can be verified with a computation of its derivative. Hence, as $0\leq q^{2k}\leq q^{2k-1}\leq 1$ it follows that
\begin{align*}
k(q)^n &=4^n\sqrt{q^n}\left(\prod_{k=1}^{\infty}
\left(\frac{1+q^{2k}}{1+q^{2k-1}}\right)^{n}\right)^4 \\
&\leq 4^n\sqrt{q^n}\left(\prod_{k=1}^{\infty}
\left(\frac{1+q^{2kn}}{1+q^{(2k-1)n}}\right)\right)^4
=4^{n-1}k(q^n).
\end{align*}
\end{proof}

Now we are ready to put things together and prove Theorem~\ref{hyper}.
\begin{proof}[Proof of Theorem~\ref{hyper}]
Consider a fixed pair of eigenvalues $a\in\sigma\brac{A}$ and $b\in\sigma\brac{B}$ connected by a certain eigenvalue curve $z:\Ibrac{0,1}\ra\C$ of $A_t := (1-t)A + tB$, with $z\brac{0} = a$ and $z\brac{1}=b$. Applying the inequality (\ref{equ:inequ}) we can estimate
\begin{align*}
\frac{1}{\Norm{A_t-A}{}} &\leq \Norm{(z\brac{t}\one - A)^{-1}}{}= \Norm{(z\brac{t}\one - A)^{-1}(\one - \overline{z\brac{t}}A)(\one - \overline{z\brac{t}}A)^{-1}}{} \\
&\leq \Norm{(z\brac{t}\one - A)^{-1}(\one - \overline{z\brac{t}}A)}{}\frac{1}{\lb 1-\rho\lb  B\rb\Norm{A}{}\rb}.
\end{align*}
We note that $\Norm{A}{}\leq1$ and $z\in\overline{\mathbb{D}}$ implies $\Norm{(1-\bar{z}A)^{-1}(z-A)}{}\leq1$~\cite[Formulas 4.12-4.14]{Nagy}. From Lemma~\ref{ratmod}, Equation~\eqref{inverse} it follows that
\begin{align*}
\Norm{(z\brac{t}\one - A)^{-1}(\one - \overline{z\brac{t}}A)}{}\leq \prod^{\abs{m}}_{i=1} \Abs{\frac{1-\overline{z\brac{t}}\lambda_i}{z\brac{t}-\lambda_i}},
\end{align*}
where $\{\lambda_i\}_{i=1,...,\abs{m}}$ are the zeros of the minimal polynomial of $A$. The theorem now follows from an application of Lemma~\ref{irgendson} and Lemma~\ref{hasest}. We choose $q$ with $k(q)=\Abs{\frac{a-b}{1-\bar{a}b}}$ and there is $t^*\in[0,1]$ such that
\begin{align*}
\frac{1}{\Norm{A_{t^*}-A}{}}\leq\frac{1}{{\lb 1-\rho\lb  B\rb\Norm{A}{}\rb}}\frac{1}{\sqrt{k(q^{2\abs{m}})}}\leq\frac{2^{2\abs{m}-1}}{{\lb 1-\rho\lb  B\rb\Norm{A}{}\rb}}\Abs{\frac{1-\bar{a}b}{a-b}}^{\abs{m}}.
\end{align*}
\end{proof}
Some remarks concerning the proof of Theorem~\ref{hyper} are in order.
\begin{enumerate}
\item \label{vomTeufel} We emphasize that one can obtain the final result, Theorem~\ref{hyper}, without relying on Chebychev-Blaschke interpolation (Lemma~\ref{lem:BlaschkeInt}). Going back to Equation~\eqref{pitty} we can find~\cite[(1.11)]{Garnett81}
\begin{align*}
\min_{s_i\in \lb -1,1\rb}& \max_{s\in\Ibrac{0,\Abs{C(a,b)}}}\Abs{\prod^{n}_{i=1}\frac{s-s_i}{1-ss_i}}=\min_{s_i\in \lb 0,1\rb} \max_{s\in\Ibrac{0,\Abs{C(a,b)}}}\Abs{\prod^{n}_{i=1}\frac{s-s_i}{1-ss_i}}\\
&\geq\min_{s_i\in \lb 0,1\rb} \max_{s\in\Ibrac{0,\Abs{C(a,b)}}}\prod^{n}_{i=1}\abs{s-s_i}\geq{\abs{C(a,b)}}^n\frac{1}{2^{2n-1}}.
\end{align*}

Lemma~\ref{irgendson} is stronger than this bound and it is the analogue of the interpolation result~\eqref{equ:Cheb}. It is natural to ask by how much the estimates differ, i.e.~how much is lost when applying Lemma~\ref{hasest}.
In the context under consideration the disadvantage is probably small. Figure~\ref{KvonQ} shows that for $q$ such that $k(q)\approx1$ the left- and right-hand side in Lemma~\ref{hasest} differ significantly, at least for small enough $n$. However, for most applications (e.g. when $B$ is a small perturbation of $A$) we are interested in the region, where $k(q)=\abs{C(a,b)}<<1$. In view of Figure~\ref{KvonQ} we expect that in this region the blue and red curve practically match.

\item The resolvent estimate used in our proof can be improved. For example (if $\Norm{A}{}\leq1$) it holds that~\cite[Corollary III.3]{OlegRes} (see also~\cite{SimonDavies}) $$\Norm{(z\one-A)^{-1}}{}\leq \frac{\cot{(\pi/(4\abs{m}))}}{\min_{\lambda_i\in\sigma(A)}\abs{1-\bar{\lambda}_iz}}\prod^{\abs{m}}_{i=1} \Abs{\frac{1-\bar{\lambda}_iz}{z-\lambda_i}}.$$

This estimate is sharp for $\abs{z}=1$ but we cannot leverage it. It is also possible to improve this bound and derive (a more complicated) estimate that is optimal for $z\not\in\sigma(A)$~\cite{Rachich}. Bounds of this type demonstrate that Blaschke products occur naturally in the context of spectral variation.
\item Our theorem can be seen as a consequence of the following fact. For any $A\in\cM_n$ with $\Norm{A}{}\leq1$ and any curve $z:[0,1]\rightarrow\mathbb{D}$ with $z(0)=a,\ z(1)=b$ there is $t^*\in[0,1]$ so that
$$\Norm{(z\brac{t^*}\one - A)^{-1}(\one - \overline{z\brac{t^*}}A)}{}\leq\frac{1}{\sqrt{k(q^{2\abs{m}})}},$$
where $q$ is chosen with $k(q)=\Abs{\frac{a-b}{1-\bar{a}b}}$.
The estimate is \emph{optimal} in the following sense. Let $z$ be a geodesic curve joining $a$ and $b$ and let $A$ be a model matrix~\eqref{modelmat} whose eigenvalues are located on $z$ between $a$ and $b$. Then
\begin{align*}
\Norm{(z\brac{t}\one - A)^{-1}(\one - \overline{z\brac{t}}A)}{}=
\prod_{i=1}^{\abs{m}} \Abs{\frac{1-\Abs{C(a,b)}^2t_i t}{\Abs{C(a,b)}\brac{t - t_i}}}
\end{align*}
for certain $t,t_i\in[0,1]$. By adjusting the eigenvalues of $A$ we can achieve that $\{\pm\sqrt{t_i}\}_i$ are the zeroes of a Chebychev-Blaschke product of degree $2\abs{m}$. In this case there is $t^*$ such that the left-hand side equals $\frac{1}{\sqrt{k(q^{2\abs{m}})}}$ with $k(q)=\Abs{\frac{a-b}{1-\bar{a}b}}$.

\item Krause's approach to spectral variation bounds~\cite{Krause199473} can probably be used to improve the bound from Theorem~\ref{hyper}. The main difficulty seems to lie in the solution of the more general interpolation
\begin{align*}
\min_{B_{n-r}\in\blaschkes_{n-r}}\max_{t\in[0,1]} t^{r}|B_{n-r}\brac{t}|
\end{align*} 
for arbitrary $r\leq n$.
\end{enumerate}

\section{Discussion}
\label{discussion}

In this section we compare the hyperbolic bound from Theorem~\ref{hyper} to the classical estimate \eqref{equ:GenBound}. For convenience we shall abbreviate
\begin{align*}
r_h := \frac{2^{2-\frac{1}{n}}}{\lb 1-\norm{A}{}\Norm{B}{}\rb^{\frac{1}{n}}}\Norm{A-B}{}^{\frac{1}{n}}\quad\textnormal{and}\quad r_e := C_{n}\lb \norm{A}{}+\norm{B}{}\rb^{1-\frac{1}{n}}\Norm{A-B}{}^{\frac{1}{n}}.
\end{align*} 
In Subsection~\ref{PertTheory} we consider small perturbations of $A$ with $\Norm{A}{}<1$ and use the two bounds to locate the eigenvalues of the perturbed matrix. In Subsection~\ref{improveBound} we drop the assumption $\Norm{A}{},\Norm{B}{} < 1$ by choosing a suitable normalization. This leads to improved $C_n$ in the strongest Euclidean spectral variation bound~\cite{Krause199473} if $\Norm{A-B}{}$ is small enough.

\subsection{Perturbation theory}\label{PertTheory}

Let $A\in\cM_n$ with $\norm{A}{}< 1$ and denote by $E$ a small perturbation of $A$, i.e.~$E\in\cM_n$ with $\norm{E}{}=\epsilon$ such that $\norm{A}{}+\epsilon <1$. Let $a\in\mathbb{D}$ denote an eigenvalue of $A$ and let $\epsilon$ be chosen so small that $r_h,r_e < 1$ holds. By Theorem \ref{hyper} there is an eigenvalue of $B=A+E$ in the hyperbolic disk
\begin{align*}
\left\{z\in \mathbb{D}\hspace*{0.1cm}|\: \Abs{\frac{a-z}{1-\overline{a}z}}<r_h\right\} = \{z\in \mathbb{D}\hspace*{0.1cm}|\: \Abs{C-z}<R\},
\end{align*}
where $C\in\mathbb{D}$ and $R\in\lbr 0,1\rbr$ are according to \eqref{equ:HypDisc}. The Euclidean radius $R$ of a disk with fixed hyperbolic radius $r_h$ becomes smaller, when $\abs{a}$ gets larger. For this reason a hyperbolic estimate gets stronger for eigenvalues of larger modulus, compare Figure \ref{fig:CompBh}. Note that if $|a-C|+R\leq r_e$ holds, then the hyperbolic disk is contained in the Euclidean disk $\{z\in \mathbb{D}|\: \Abs{a-z}<r_e\}$. By direct computation this condition is equivalent to
\begin{align}
1-\abs{a}^2\leq \frac{r_e}{r_h}(1-r_h\abs{a}),\label{equ:HradiusVsEradius}
\end{align}
which can be satisfied even if $r_h>r_e$, i.e.~the numerical value of the estimate from Theorem~\ref{hyper} is larger than the one from Equation~\ref{equ:GenBound}. Inserting $r_e$ and $r_h$ and using $\abs{a}\leq 1$ this is certainly fulfilled if
\begin{align*}
1 \leq \frac{C_n}{2^{2-\frac{1}{n}}}(\norm{A}{}+\norm{B}{})^{1-\frac{1}{n}}(1-\norm{A}{}\norm{B}{})^{\frac{1}{n}}(1-\frac{2^{2-\frac{1}{n}}}{(1-\norm{A}{}\norm{B}{})^\frac{1}{n}}\epsilon^{\frac{1}{n}}).
\end{align*}
For simplicity we shall assume that $\norm{A}{}=\norm{B}{}$ but a similar result holds for small $\epsilon$ as $\norm{A}{}-\epsilon\leq \norm{B}{}\leq \norm{A}{}+\epsilon$. Then it is sufficient if $\epsilon$ is chosen so that
\begin{align*}
\lb\frac{\epsilon}{2}\rb^{\frac{1}{n}} \leq \frac{1}{4}(1-\norm{A}{}^2)^\frac{1}{n} - \frac{1}{2C_n\norm{A}{}}.
\end{align*}
Inserting $C_n = 2^{2-\frac{1}{n}}$ we see that the right-hand side becomes positive for $\norm{A}{}>\frac{1}{2}$ and sufficiently large $n$. Hence in this case we can always find $\epsilon=\epsilon(n)$ so that the estimate holds true. If $n=5$ is chosen fixed then the right-hand side above is positive if $0.65\leq \norm{A}{}\leq 0.95$ and the inequality is fulfilled if $\epsilon=\epsilon(\Norm{A}{})$ is chosen small, compare Figure \ref{fig:CompBh}. We can also compare our estimate to~\eqref{equ:GenBound} with the best constant $C_n = \frac{16}{3\sqrt{3}}$ ~\cite{Krause199473}. For sufficiently large $n$ and $\norm{A}{}>\frac{3\sqrt{3}}{8}\approx 0.65$, again by appropriate choice of $\epsilon$, Theorem~\ref{hyper} is stronger than~\eqref{equ:GenBound}. Figure \ref{fig:CompKr} depicts the case $n=12$ with $C_{12} = 2.6543$ taken from~\cite[Table 1]{Krause199473}.

\begin{figure}
\centering
\begin{minipage}{.5\textwidth}
  \centering
  \includegraphics[width=6.5cm]{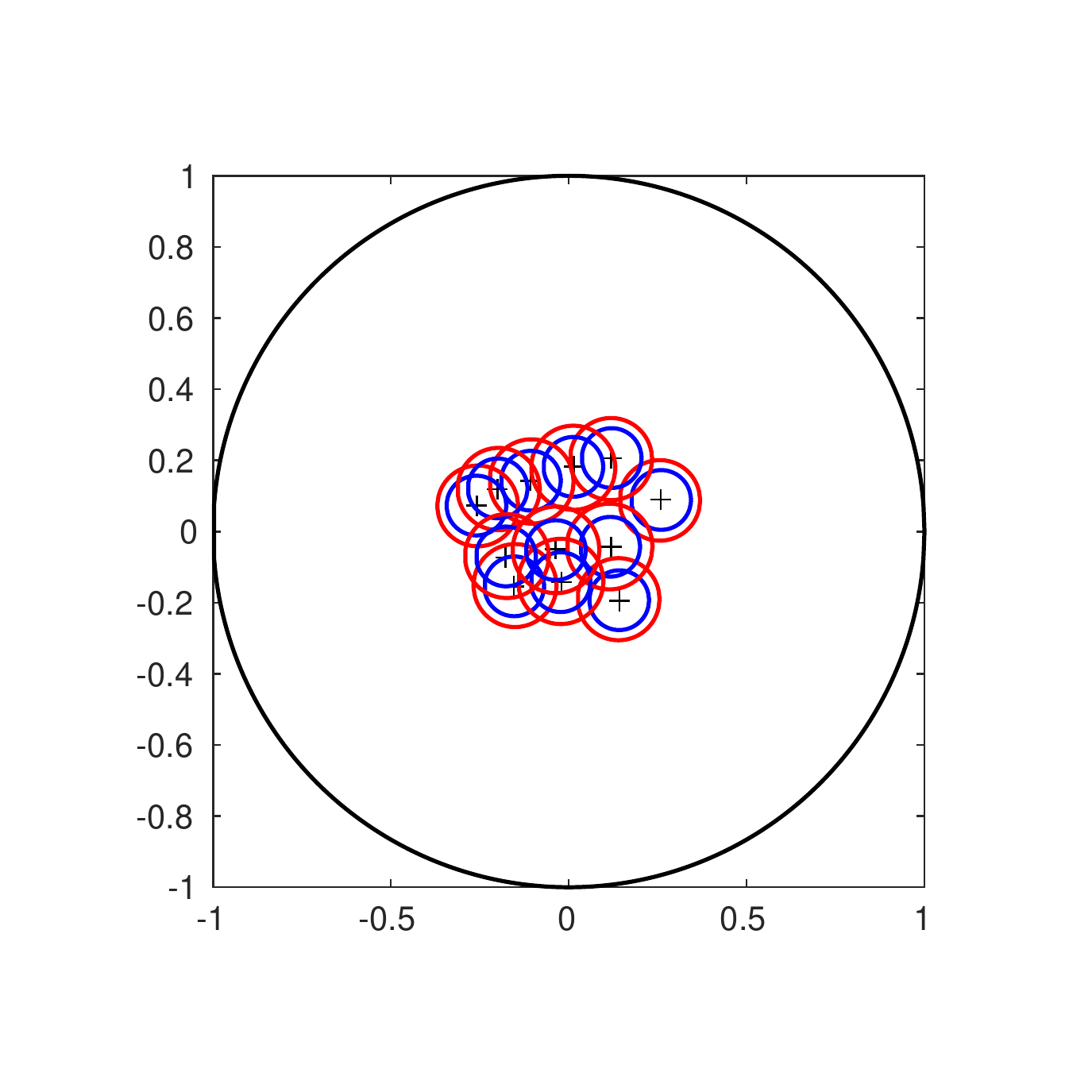}
  \vspace*{-1.5cm}
  \captionof*{figure}{(a) $\norm{A}{}=0.5$}
\end{minipage}%
\begin{minipage}{.5\textwidth}
  \centering
\includegraphics[width=6.5cm]{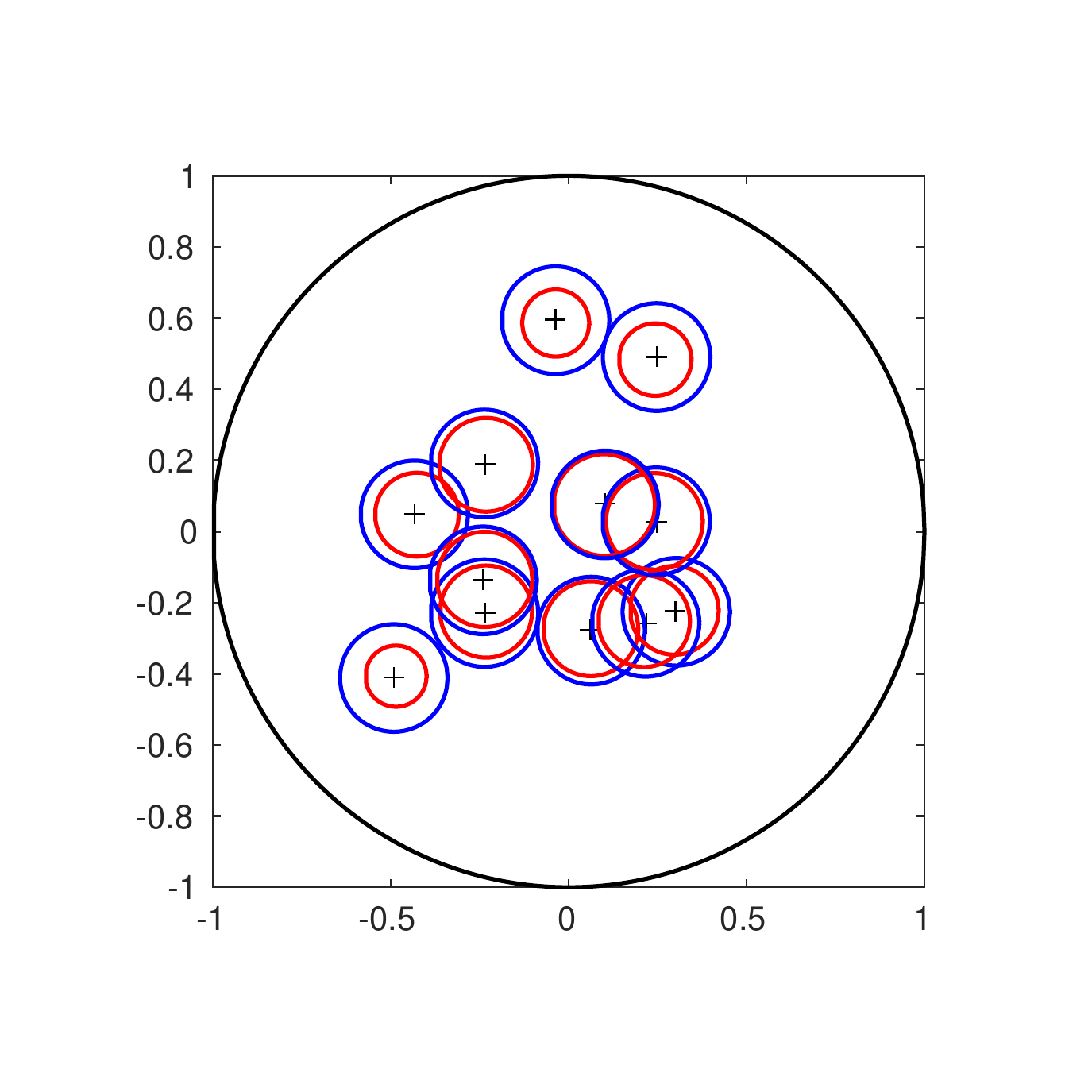}
\vspace*{-1.5cm}
  \captionof*{figure}{(b) $\norm{A}{}=0.95$}
\end{minipage}
\caption{Localization of eigenvalues of a perturbed $12\times 12$ matrix $B=A+E$ for randomly chosen $A$ and $\norm{E}{}\leq 10^{-18}$ with respect to the spectrum of $A$. Comparison between Theorem \ref{hyper}~(red) and the bound \eqref{equ:GenBound}~(blue) with $C_{12} = 2.6543$~.}
\label{fig:CompKr}
\end{figure}

On the other hand, if under assumption $\norm{A}{}=\norm{B}{}$ we have
\begin{align*}
1-\frac{C_n}{2}\norm{A}{}^{1-\frac{1}{n}}\lb 1-\norm{A}{}^2\rb^{\frac{1}{n}} > \norm{A}{}^2
\end{align*}
then \eqref{equ:HradiusVsEradius} cannot be fulfilled. For $C_n = 2^{2-\frac{1}{n}}$ this holds true for sufficiently large $n$ and $\norm{A}{}<\sqrt{2}-1$, see Figure \ref{fig:CompBh}.

\subsection{Improving the Euclidean bound}\label{improveBound}

In Section \ref{PertTheory} we have shown how our hyperbolic spectral variation bound improves on known estimates for certain choices of $\Norm{A}{}<1$ and if $\Norm{A-B}{}$ is small enough. Here, we drop the assumption $\Norm{A}{}<1$ and derive an Euclidean spectral variation bound with an improved constant as compared to~\cite{Krause199473}. As in~\cite{Krause199473} we set 
$M_2 := \max\lset\norm{A}{},\norm{B}{}\rset$ for our analysis.
\begin{corollary}\label{kaputze}
For $A,B\in\cM_n$ with $M_2 := \max\lset\norm{A}{},\norm{B}{}\rset$ and distance 
\begin{align}
\Norm{A-B}{}\leq \lb\frac{1}{2M_2}\rb^{n-1} \lb\frac{n+1}{n-1}\rb^n \alpha^n_n\min_{a\in\sigma(A)\setminus\lset 0\rset}\abs{a}^n
\label{DistanceUpper}
\end{align}
we get
\begin{align}
d_E\lb\sigma\lb A\rb,\sigma\lb B\rb\rb \leq \frac{1}{\alpha_n}\lb 2M_2\rb^{1-\frac{1}{n}}\Norm{A-B}{}^{\frac{1}{n}},\label{KrauseBound}
\end{align}
where
\begin{align*}
\alpha_n := \frac{1}{2}\lb\frac{2}{\sqrt{n^2-1}}\rb^\frac{1}{n}\sqrt{\frac{n-1}{n+1}}. 
\end{align*} 
\label{cor:Improved}
\end{corollary} 

Table \ref{table:Values} shows some values for $\frac{1}{\alpha_n}$. For $n\geq 12$ these values are smaller than the bounds in~\cite{Krause199473}. It is interesting to note, that the sequence of $\frac{1}{\alpha_n}$ is decreasingly converging to $2$, which is the optimal asymptotic behaviour as this quantity cannot be smaller than $2^{1-\frac{1}{n}}$~\cite{Bhatia1}. Corollary~\ref{kaputze} improves on the best Euclidean spectral variation bounds if the distance $\norm{A-B}{}$ is small enough. More specifically direct computation shows that we require $\norm{A-B}{}\lesssim \lb\frac{|a^*|}{4M_2}\rb^{n-1}\frac{|a^*|e}{\sqrt{n^2-1}}$ for $n$ large enough, where $a^*$ denotes the (in modulus) smallest non-zero eigenvalue of $A$. Note that the right-hand side goes to zero exponentially fast in $n$ at fixed $a^*$, which means that in practice
$\norm{A-B}{}$ would have to be extremely small.

\begin{proof}

First consider the case $M_2<1$. By choosing a constant $C$ such that
\begin{align}
\frac{2^{2-\frac{1}{n}}}{\lb 1-M^2_2\rb^{\frac{1}{n}}}\Norm{A-B}{}^{\frac{1}{n}} = C\lb 2M_2\rb^{1-\frac{1}{n}}\Norm{A-B}{}^{\frac{1}{n}}\label{cequ}
\end{align}
we have $r_h:=\frac{2^{2-\frac{1}{n}}}{\lb 1-\norm{A}{}\Norm{B}{}\rb^{\frac{1}{n}}}\Norm{A-B}{}^{\frac{1}{n}}\leq C\lb 2M_2\rb^{1-\frac{1}{n}}\Norm{A-B}{}^{\frac{1}{n}}$. Assuming $r_h\leq \min_{a\in\sigma(A)\setminus\lset 0\rset}\abs{a}$ inequality \eqref{equ:HradiusVsEradius} is fulfilled for for any $a\in\sigma(A)$ and any radius $r_e\geq r_h$. By the discussion preceding \eqref{equ:HradiusVsEradius} this shows that the Euclidean disk with radius $r_e = C\lb 2M_2\rb^{1-\frac{1}{n}}\Norm{A-B}{}^{\frac{1}{n}}$ around any $a\in\sigma\lb A\rb$ contains the hyperbolic disk with radius $r_h$ around $a\in\sigma\lb A\rb$ in which an eigenvalue of $B$ is located by Theorem \ref{hyper}. This shows that if $r_h\leq C\lb 2M_2\rb^{1-\frac{1}{n}}\Norm{A-B}{}^{\frac{1}{n}}\leq \min_{a\in\sigma(A)\setminus\lset 0\rset}\abs{a}$, then we can bound the Euclidean optimal matching distance by
\begin{align}
d_E\lb\sigma\lb A\rb,\sigma\lb B\rb\rb\leq C\lb 2M_2\rb^{1-\frac{1}{n}}\Norm{A-B}{}^{\frac{1}{n}}.
\label{equ:bound24}
\end{align} 
Next we want to find the value of $M_2$ such that $C$ in the above inequality is smallest. Note that \eqref{cequ} is equivalent to
\begin{align*}
\frac{C}{2}M_2^{1-\frac{1}{n}}\lb 1-M_2^2\rb^{\frac{1}{n}} = 1.
\end{align*}
The maximum
\begin{align*}
\max_{x\in\lbr 0,1\rb} \frac{1}{2}\lb 1-x^2\rb^\frac{1}{n} x^{1-\frac{1}{n}} = \alpha_n
\end{align*}
is attained for $x_\text{max} := \sqrt{\frac{n-1}{n+1}}$. It follows that we can choose $C=\frac{1}{\alpha_n}$ for $M_2 = x_\text{max}$. But $M_2$ can be always set to this value by suitable normalization. 

For arbitrary $M_2$ and $\Norm{A-B}{}$ by \eqref{equ:bound24} it holds that
\begin{align*}
d_E\lb\sigma\lb A\rb,\sigma\lb B\rb\rb &= \frac{M_2}{x_\text{max}} d_E\lb\sigma\lb \frac{x_\text{max}}{M_2}A\rb,\sigma\lb \frac{x_\text{max}}{M_2}B\rb\rb \\
&\leq\frac{M_2}{x_\text{max}}\frac{1}{\alpha_n}\lb 2 x_\text{max}\rb^{1-\frac{1}{n}}\lb\frac{x_\text{max}}{M_2}\Norm{A-B}{}\rb^{\frac{1}{n}} \\
&=\frac{1}{\alpha_n}\lb 2M_2\rb^{1-\frac{1}{n}}\Norm{A-B}{}^{\frac{1}{n}}
\end{align*}
provided that $\frac{1}{\alpha_n}\lb 2x_{\text{max}}\rb^{1-\frac{1}{n}}\lb\frac{x_{\text{max}}}{M_2}\rb^{\frac{1}{n}}\Norm{A-B}{}^{\frac{1}{n}}\leq \frac{x_{\text{max}}}{M_2} \min_{a\in\sigma(A)\setminus\lset 0\rset}\abs{a}$, which is equivalent to \eqref{DistanceUpper}.

\end{proof}

\begin{table}
\begin{tabular}{l|*{8}{|c}|}
\hspace*{0.1cm}$n$ & \hspace*{0.2cm}1\hspace*{0.2cm} & \hspace*{0.2cm}2\hspace*{0.2cm} & \hspace*{0.2cm}3\hspace*{0.2cm} & \hspace*{0.2cm}4\hspace*{0.2cm} & \hspace*{0.2cm}5\hspace*{0.2cm} & \hspace*{0.2cm}6\hspace*{0.2cm} & \hspace*{0.2cm}7\hspace*{0.2cm} & \hspace*{0.2cm}8\hspace*{0.2cm} \\
\hline
$\frac{1}{\alpha_n}$ & 2 & 3.2237 & 3.1748 & 3.0458 & 2.9302 & 2.8353 & 2.7579 & 2.6942 
\end{tabular}$~\cdots$\vspace*{0.3cm}
\hspace*{0.5cm}$\cdots$\begin{tabular}{c*{7}{|c}}
\hspace*{0.2cm}9\hspace*{0.2cm} & \hspace*{0.2cm}10\hspace*{0.2cm} & \hspace*{0.2cm}11\hspace*{0.2cm} & \hspace*{0.2cm}12\hspace*{0.2cm}& \hspace*{0.2cm}100\hspace*{0.2cm}& \hspace*{0.2cm}1000\hspace*{0.2cm}& \hspace*{0.2cm}$\infty$\hspace*{0.2cm}  \\
\hline
2.6410	& 2.5959 & 2.5572 & 2.5236 & 2.101 & 2.0145 & 2
\end{tabular}
\caption{Some values of $\frac{1}{\alpha_n}$.}
\label{table:Values}
\end{table}

\section*{Acknowledgements}
OS and AMH acknowledge financial support by the Elite Network of Bavaria (ENB) project
QCCC and the CHIST-ERA/BMBF project CQC. OS acknowledges financial support from European Union under project QALGO (Grant Agreement No. 600700). This work was made possible partially through the support of grant
\#48322 from the John Templeton Foundation. The opinions
expressed in this publication are those of the authors and do not
necessarily reflect the views of the John Templeton Foundation. We are thankful to Michael M. Wolf for creating conditions that made this work possible.

\section*{References}

\bibliographystyle{elsarticle-harv}

\appendix
\section{Perpendicular projection in the Poincar\'e disk model}
\label{sec:Appendix}

\begin{lemma}
Let $\Gamma$ be a geodesic curve in the Poincar\'{e} disk model of hyperbolic geometry and denote by $p_\Gamma:\mathbb{D}\ra \Gamma$ the perpendicular projection onto $\Gamma$. Then for any $z,w\in\mathbb{D}$ 
\begin{align*}
\Abs{\frac{z-w}{1-\bar{z}w}}\geq \Abs{\frac{p_\Gamma(z)-p_\Gamma(w)}{1-\overline{p_\Gamma(z)}p_\Gamma(w)}}.
\end{align*}
\end{lemma}

Due to lack of a reference we provide a proof for this fact.

\begin{proof}
For any point $z\in\mathbb{D}\setminus{\Gamma}$ there is a unique geodesic line $\Lambda_z$ through $z$ intersecting $\Gamma$ in a right angle~\cite[p. 305]{Brannan}. We define the perpendicular projection $p_\Gamma(z)$ as the intersection of $\Lambda_z$ and $\Gamma$.
It is well known that M\"obius transformations map geodesics to geodesics, preserve angles and pseudo-hyperbolic distance in $\mathbb{D}$. We can thus apply a map of the form

$$\zeta\mapsto e^{i\theta}\frac{\zeta-p_\Gamma(z)}{1-\overline{p_\Gamma(z)}\zeta}\quad \theta\in[0,2\pi)$$
to restrict our discussion to 
\begin{align*}
p_\Gamma(z)=0, \quad \Gamma = \lbr -1,1\rbr, \quad \Lambda_z = i\lbr -1,1\rbr.
\end{align*}
Under this map $\Lambda_w$ is mapped to a geodesic that orthogonally intersects $\Gamma$ in $p_\Gamma(w)$, see Figure \ref{fig:purp}.  
\begin{figure}[t]
\centering
\includegraphics[width=6cm]{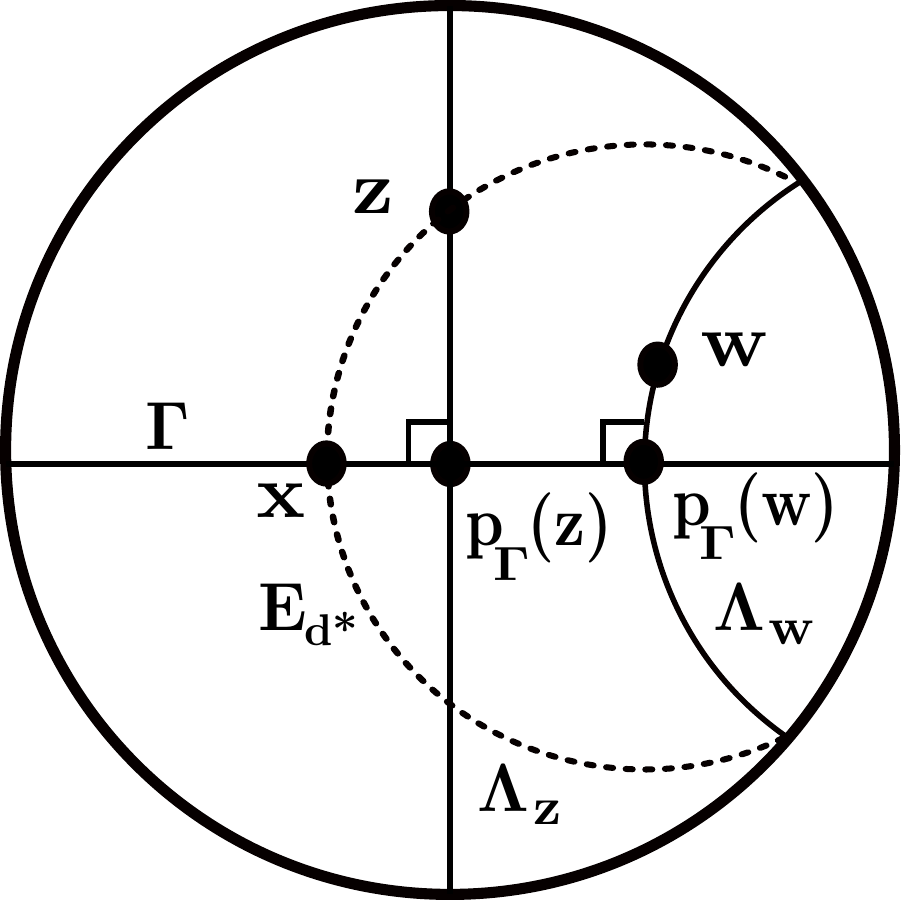}
\caption{Perpendicular projections of $z,w\in\mathbb{D}$ onto the geodesic $\Gamma$. Dashed line is the set of points with distance $d^*$ from the geodesic $\Lambda_w$.}
\label{fig:purp}
\end{figure}
Consider sets of points that are of hyperbolic distance $d$ from $\Lambda_w$
\begin{align*}
E_d = \left\{ y\in\mathbb{D}\hspace*{0.1cm}| \Abs{\frac{y-p_{\Lambda_w}(y)}{1-\overline{y}p_{\Lambda_w}(y)}} = d \right\}.
\end{align*}
It can be shown that these sets are circular arcs intersecting $\partial\mathbb{D}$ in the same points as $\Lambda_w$~\cite[p. 313]{Brannan}. There is $d^*$ with $z\in E_{d^*}$, see Figure \ref{fig:purp}, and by the hyperbolic Theorem of Pythagoras~\cite[p.307]{Brannan} we have
\begin{align*}
d^* \leq \Abs{\frac{z-w}{1-\bar{z}w}}.
\end{align*}
Let $x\in E_{d^*}\cap \lbr -1,1\rbr$ denote the intersection of $E_{d^*}$ with $\Gamma$. By convexity of Euclidean disks it is clear that $x\in\lbr -1,0\rb$. We conclude that
\begin{align*}
\Abs{p_\Gamma(w)}\leq\Abs{\frac{x-p_\Gamma(w)}{1-xp_\Gamma(w)}}=d^*\leq\Abs{\frac{z-w}{1-\bar{z}w}}.
\end{align*}
\end{proof}

\end{document}